\documentclass[reqno,dvips]{amsart}

\usepackage{mathrsfs}
\usepackage{amsmath}
\usepackage{amssymb}
\usepackage{cite}
\usepackage{comment}
\usepackage{latexsym}
\usepackage{graphicx}

\usepackage{color}

\theoremstyle{plain}
\newtheorem{thm}{Theorem}[section]

\newtheorem{rmk}[thm]{Remark}

\def\D{\mathrm{D}}

\def\S{\mathscr{S}}
\def\T{\mathscr{T}}

\def\d{\mathrm{d}}
\def\e{\mathrm{e}}

\def\Cset{\mathbb{C}}

\def\Nset{\mathbb{N}}
\def\Qset{\mathbb{Q}}
\def\Rset{\mathbb{R}}
\def\Sset{\mathbb{S}}
\def\Tset{\mathbb{T}}
\def\Zset{\mathbb{Z}}

\def\epsilon{\varepsilon}

\def\epsilon{\varepsilon}

\DeclareMathOperator{\arctanh}{arctanh}
\DeclareMathOperator{\im}{Im}

\makeatletter
 \@addtoreset{equation}{section}
\makeatother

\begin{document}


\title[A new proof of Poincar\'e's result]%
{A new proof of Poincar\'e's result on 
  the restricted three-body problem}
\thanks{This work was partially supported by the JSPS KAKENHI Grant Number JP17H02859.}

\author{Kazuyuki Yagasaki}

\address{Department of Applied Mathematics and Physics, Graduate School of Informatics,
Kyoto University, Yoshida-Honmachi, Sakyo-ku, Kyoto 606-8501, JAPAN}
\email{yagasaki@amp.i.kyoto-u.ac.jp}

\date{\today}

\begin{abstract}
The problem of nonintegrability of the circular restricted three-body problem
 is very classical and important in dynamical systems.
In the first volume of his masterpieces,
 Henri Poincar\'e showed  the nonexistence of a real-analytic first integral
 which is functionally independent of the Hamiltonian
 and real-analytic in a small parameter representing the mass ratio
 as well as in the state variables,
 in both the planar and spatial cases.
However, his proof was very complicated and unclear.
In this paper,
 we give a new and simple proof of a very similar result for both the planar and spatial cases,
 using an approach which the author developed recently for nearly integrable systems.
\end{abstract}
\maketitle


\section{Introduction}

\begin{figure}[t]
\includegraphics[scale=0.8]{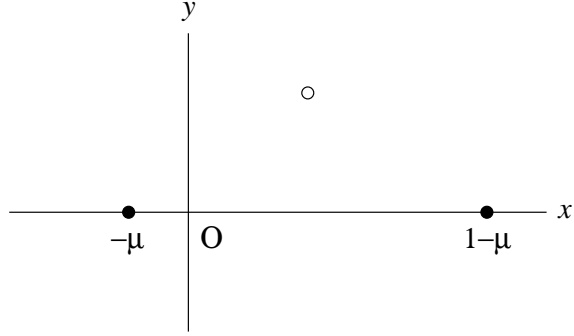}
\caption{Configuration of the circular restricted three-body problem in the rotational frame.
\label{fig:1a}}
\end{figure}

In his famous memoir \cite{P90},
 which was related to a prize competition celebrating the 60th birthday of King Oscar II,
 Henri Poincar\'e studied the planar circular restricted three-body problem,
 which is written in the dimensionless form as the Hamiltonian system
\begin{equation}
\begin{split}
&
\dot{x}=p_x+y,\quad
\dot{p}_x=p_y+\frac{\partial U_2}{\partial x}(x,y),\\
&
\dot{y}=p_y-x,\quad
\dot{p}_y=-p_x+\frac{\partial U_2}{\partial y}(x,y),
\end{split}
\label{eqn:pp}
\end{equation}
where
\[
U_2(x,y)=\frac{\mu}{\sqrt{(x-1+\mu)^2+y^2}}+\frac{1-\mu}{\sqrt{(x+\mu)^2+y^2}},
\]
and discussed the nonexistence of a real-analytic first integral
 which is real-analytic in the state variables and parameter $\mu$ near $\mu=0$
 and functionally independent of its Hamiltonian
\[
H_2(x,y,p_x,p_y)=\tfrac{1}{2}(p_x^2+p_y^2)+(p_xy-p_yx)-U_2(x,y).
\]
He improved his approach significantly
 in the first volume of his masterpieces \cite{P92} published two years later:
 he considered not only the planar case \eqref{eqn:pp} but also the spatial case
\begin{equation}
\begin{split}
&
\dot{x}=p_x+y,\quad
\dot{p}_x=p_y+\frac{\partial U_3}{\partial x}(x,y,z),\\
&
\dot{y}=p_y-x,\quad
\dot{p}_y=-p_x+\frac{\partial U_3}{\partial y}(x,y,z),\\
&
\dot{z}=p_z,\quad
\dot{p}_z=\frac{\partial U_3}{\partial z}(x,y,z),
\end{split}
\label{eqn:sp}
\end{equation}
where
\[
U_3(x,y,z)=\frac{\mu}{\sqrt{(x-1+\mu)^2+y^2+z^2}}+\frac{1-\mu}{\sqrt{(x+\mu)^2+y^2+z^2}},
\]
and showed  the nonexistence of such first integrals.
The system \eqref{eqn:sp} is also Hamiltonian with the Hamiltonian
\[
H_3(x,y,z,p_x,p_y,p_z)=\tfrac{1}{2}(p_x^2+p_y^2+p_z^2)+(p_xy-p_yx)-U_3(x,y,z)
\]
In \eqref{eqn:pp} and \eqref{eqn:sp},
 the two primary bodies with mass $\mu$ and $1-\mu$
 remain at $(1-\mu,0)$ and $(-\mu,0)$, respectively, on the $xy$-plane in the rotational frame,
 and the third massless body is subjected to the gravitational forces from them
 moves under the assumption that the primaries rotate counterclockwise
 on the circles about their common center of mass at the origin
 in the inertial coordinate frame (see Fig.~\ref{fig:1a}).
See, e.g., Section~4.1 of \cite{MO17}
 for more details on the derivation and physical meaning of \eqref{eqn:pp} and \eqref{eqn:sp}.
The results of \cite{P92} were also explained in \cite{AKN06,K83,K96,W37}.
See \cite{B96} for an account of his work from mathematical and historical perspectives.
His result can be stated as follows.

\begin{thm}
\label{thm:P}
The circular restricted three-body problems \eqref{eqn:pp} and \eqref{eqn:sp}
 have no real-analytic first integral which is independent of the Hamiltonian
 and depend real-analytically on $\mu$ near $\mu=0$. 
\end{thm}

When $\mu=0$, Eqs.~\eqref{eqn:pp} and \eqref{eqn:sp} are integrable
 and can be transformed to simple Hamiltonian systems
 in action-angle coordinates (see Eqs.~\eqref{eqn:naasys}, \eqref{eqn:ppHD} and \eqref{eqn:spHD})
 by the Delaunay elements (see Eqs.~\eqref{eqn:ppDe} and \eqref{eqn:spDe}).
However, in the coordinates,
 the perturbation terms of $O(\mu)$ are very complicated.
Actually, he devoted the whole of Chapter~6, 
 the length of which is 55 pages in the English translated version
 and 66 pages in the French original one,
 in \cite{P92} to discussions on the perturbation terms.
The long and complicated computations are  very difficult to understand.
Here we prove a result similar to Theorem~\ref{thm:P}
 using a different approach which does not rely on the complicated perturbation terms.
The planar and spatial cases are discussed in Sections~3 and 4, respectively.
See Theorems~\ref{thm:pp} and \ref{thm:sp} below for the precise statements especially.
Our results are more general than Theorem~\ref{thm:P} in some meaning
 since it says that the planar and spatial problems \eqref{eqn:pp} and \eqref{eqn:sp}
 are not only analytically but also meromorphically nonintegarble.
Moreover, they are proven to be nonintegrable even in some resonant planar elliptic orbits.
See also Remarks~\ref{rmk:3a} and \ref{rmk:4a}.
However, in another meaning,
 Theorem~\ref{thm:P} is more general than ours
 since it says that they are real-analytically nonintegrable
 and guarantees no additional first integral in the spatial problem \eqref{eqn:sp}.

Our basic tool to obtain the result is a technique developed in \cite{Y21a}
 for determining  whether systems of the form
\begin{equation}
\dot{I}=\epsilon h(I,\theta;\epsilon),\quad
\dot{\theta}=\omega(I)+\epsilon g(I,\theta;\epsilon),\quad
(I,\theta)\in\Rset^\ell\times\Tset^m,
\label{eqn:aasys}
\end{equation}
are not meromorphically integrable in the Bogoyavlenskij sense \cite{B98},
 which is an extended concept of integrability to non-Hamiltonian systems,
 such that the first integrals and commutative vector fields,
 the existence of which are required for their integrability,
 depend meromorphically on $\epsilon$ near $\epsilon=0$,
 like the result of Poincar\'e \cite{P90,P92} stated above with $\epsilon=\mu$,
 where $\ell,m\in\Nset$, $\epsilon$ is a small parameter such that $0<|\epsilon|\ll 1$,
 and $\omega:\Rset^\ell\to\Rset^m$,
 $h:\Rset^\ell\times\Tset^m\times\Rset\to\Rset^\ell$
 and $g:\Rset^\ell\times\Tset^m\times\Rset\to\Rset^m$ are meromorphic in their arguments.
The technique is based on generalized versions due to Ayoul and Zung \cite{AZ10}
 of the Morales-Ramis theory \cite{M99,MR01}
 and its extension, the Morales-Ramis-Sim\'o theory \cite{MRS07},
 and briefly reviewed in a necessary context in Section~2 for the reader's convenience.
The system~\eqref{eqn:aasys} is Hamiltonian if $\ell=m$ as well as $\epsilon=0$ or
\[
\D_I h(I,\theta;\epsilon)\equiv-\D_\theta g(I,\theta;\epsilon),
\]
and non-Hamiltonian if not.
If a system is integrable in the Bogoyavlenskij sense
 and has $\ell$ functionally independent first integrals
 whose level set has a connected compact set,
 then it is transformed to \eqref{eqn:aasys} with $\epsilon=0$
 on the connected compact set \cite{B98,Z18}.
Moreover, in \cite{Y21a},
 it was proven by using the technique that
 the planar and spatial problems \eqref{eqn:pp} and \eqref{eqn:sp}
 are meromorphically nonintegrable near $(x,y)=(-\mu,0),(1-\mu,0)$
 and $(x,y,z)=(-\mu,0,0),(1-\mu,0,0)$, respectively, for any $\mu\in(0,1)$.
The results of \cite{Y21a} immediately yield statements similar to ours
 but a region of the phase space in which the nonexistence of additional first integrals is proven
 is much narrower than ours.
See Theorem~\ref{thm:pp} and \ref{thm:sp} for more details.

In closing this section we give some comments
 on other recent or relatively recent remarkable progress in related topics.
First, the nonintegrability of the general three-body problem is now well understood.
Tsygvintsev \cite{T00,T01a} proved the nonintegrability
 of the general planar three-body problem
 near the Lagrangian parabolic orbits in which the three bodies form an equilateral triangle
 and move along certain parabolas, using the Ziglin method \cite{Z82}.
Boucher and Weil \cite{BW03} also obtained a similar result,
 using the Morales-Ramis theory \cite{M99,MR01}
 while the case of equal masses was proven a little earlier in \cite{B00}.
Moreover, Tsygvintsev \cite{T01b,T03,T07} proved the nonexistence
 of a single additional first integral near the Lagrangian parabolic orbits when
\[
\frac{m_1m_2+m_2m_3+m_3m_1}{(m_1+m_2+m_3)^2}\notin\{\tfrac{7}{48},\tfrac{5}{16}\},
\]
where $m_j$ represents the mass of the $j$th body for $j=1,2,3$.
Subsequently, Morales-Ruiz and Simon \cite{MS09}
 succeeded in removing the three exceptional cases and extended the result
 to the space of three or more dimensions.
Ziglin \cite{Z00} also proved the nonintegrability of the general three-body problem
 near a collinear solution
 in the space of any dimension
 when two of the three masses, say $m_1,m_2$, are nearly equal
 but neither $m_3/m_1$ nor $m_3/m_2\in\{11/12,1/4,1/24\}$.
It should be noted that
 Ziglin \cite{Z00} and Morales-Ruiz and Simon \cite{MS09}
 also discussed the general $N$-body problem.

Secondly, Guardia et al. \cite{GMS16} considered the planar problem \eqref{eqn:pp}
 and showed the occurrence of transverse intersection
 between the stable and unstable manifolds of the infinity for any $\mu\in(0,1)$
 in a region far from the primaries
 in which $r=\sqrt{x^2+y^2}$ and its conjugate momentum are sufficiently large.
This implies, e.g., by Theorem~3.10 of \cite{M73}, 
 the real-analytic nonintetgrabilty of \eqref{eqn:pp}
 as well as the existence of oscillatory motions
 such that $\limsup_{t\to\infty}r(t)=\infty$ while $\liminf_{t\to\infty}r(t)<\infty$.
Similar results were obtained much earlier
 when $\mu>0$ is sufficiently small in \cite{LS80}
 or for any $\mu\in(0,1)$ except for a certain finite number of the values in \cite{X92}.

\section{Basic Tool}
Our proofs of the main results are based on the technique developed in \cite{Y21a}
 for proving the nonintegrability of nearly integrable dynamical systems,
 as stated in Section~1.
In this section we restrict ourselves to Hamiltonian systems,
 and review the basic tool in the context.

\begin{figure}
\includegraphics[scale=0.8]{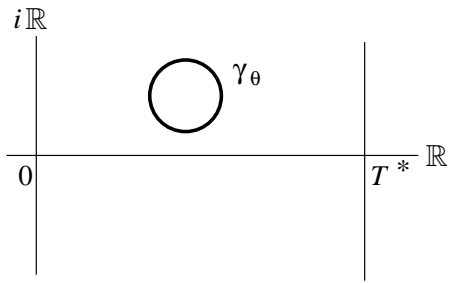}
\caption{Assumption~(A2).\label{fig:2a}}
\end{figure}

Consider $n$-degree-of-freedom Hamiltonian systems of the form 
\begin{equation}
\dot{I}=-\epsilon\D_\theta\tilde{H}(I,\theta;\epsilon),\quad
\dot{\theta}=\D H_0(I)+\epsilon\D_I\tilde{H}(I,\theta;\epsilon),\quad
(I,\theta)\in\Rset^n\times\Tset^n,
\label{eqn:naasys}
\end{equation}
where $n\ge 2$ is an integer, $\epsilon$ is a small parameter such that $0<|\epsilon|\ll 1$,
 and $H_0:\Rset^n\to\Rset$ and $\tilde{H}:\Rset^n\times\Tset^n\times\Rset\to\Rset$
 are meromorphic  in the arguments.
The Hamiltonian for \eqref{eqn:naasys} is given by $H_0(I)+\epsilon\tilde{H}(I,\theta)$.
We extend the domain of the independent variable $t$
 to a domain including $\Rset$ in $\Cset$ and do so for the dependent variables.
When $\epsilon=0$, Eq.~\eqref{eqn:naasys} becomes
\begin{equation}
\dot{I}=0,\quad
\dot{\theta}=\omega(I),
\label{eqn:naasys0}
\end{equation}
where $\omega(I)=\D H_0(I)$.
We assume the following on the unperturbed system \eqref{eqn:naasys0}:
\begin{enumerate}
\setlength{\leftskip}{-1em}
\item[\bf(A1)]
For some $I^\ast\in\Rset^n$, a resonance of multiplicity $n-1$,
\[
\dim_\Qset\langle\omega_1(I^\ast),\ldots,\omega_n(I^\ast)\rangle=1,
\]
occurs with $\omega(I^\ast)\neq 0$,
 i.e., there exists a constant $\omega^\ast>0$ such that 
\[
\frac{\omega(I^\ast)}{\omega^\ast}\in\Zset^n\setminus\{0\}.
\]
where $\omega_j(I)$ is the $j$th element of $\omega(I)$ for $j=1,\ldots,n$.
\end{enumerate}
Note that we can replace $\omega^\ast$ with $\omega^\ast/k$ for any $k\in\Nset$.
We refer
 to the $n$-dimensional torus $\T^\ast=\{(I^\ast,\theta)\mid\theta\in\Tset^n\}$ as the \emph{resonant torus}
 and to periodic orbits $(I,\theta)=(I^\ast,\omega(I^\ast)t+\theta_0)$, $\theta_0\in\Tset^n$,
 on $\T^\ast$ as the \emph{resonant periodic orbits}.
Let $T^\ast=2\pi/\omega^\ast$.
We also make the following assumption.
\begin{enumerate}
\setlength{\leftskip}{-1em}
\item[\bf(A2)]
For some $\theta\in\Tset^n$ there exists a closed loop $\gamma_\theta$
 in a domain including $(0,T^\ast)\subset\Rset$ in $\Cset$
 such that $\gamma_\theta\cap(i\Rset\cup(T^\ast+i\Rset))=\emptyset$ and
\begin{equation}
\mathscr{I}(\theta)=-\D\omega(I^\ast)\int_{\gamma_\theta}
 \D_\theta\tilde{H}(I^\ast,\omega(I^\ast)\tau+\theta;0)\d\tau
\label{eqn:A2}
\end{equation}
is not zero.
See Fig.~\ref{fig:2a}
\end{enumerate}
Note that the condition $\gamma_\theta\cap(i\Rset\cup(T^\ast+i\Rset))=\emptyset$
 is not essential in (A2), since it always holds
 by replacing $\omega^\ast$ with $\omega^\ast/k$ for sufficiently large $k\in\Nset$ if necessary.
We can prove the following theorem
 which guarantees that conditions~(A1) and (A2) is sufficient for nonintegrability of \eqref{eqn:naasys}
 in a certain meaning.

\begin{thm}
\label{thm:tool}
Let $\Gamma$ be any domain in $\Cset/T^\ast\Zset$
 containing $\Rset/T^\ast\Zset$ and $\gamma_\theta$.
Suppose that assumption~{\rm(A1)} and {\rm(A2)} hold for some $\theta=\theta_0\in\Tset^n$.
Then the system \eqref{eqn:aasys} is not meromorphically integrable 
 near the resonant periodic orbit $(I,\theta)=(I^\ast,\omega(I^\ast)\tau+\theta_0)$ with $\tau\in\Gamma$
 such that the first integrals 
 also depend meromorphically on $\epsilon$ near $\epsilon=0$,
 when the domains of the independent and dependent variables are extended to regions
 in $\Cset$ and $\Cset^n\times(\Cset/2\pi\Zset)^n$, repsectively.
Moreover, if {\rm(A2)} holds for any $\theta\in\Delta$, where $\Delta$ is a dense set in $\Tset^n$,
 then the conclusion holds for any resonant periodic orbit on the resonant torus $\T^\ast$.
\end{thm}

See Section~2 of \cite{Y21a} for a proof.
A more general result for non-Hamiltonian systems was obtained there.

\section{Planar Case}
We take $\mu$ as the small parameter $\epsilon$
 and discuss the planar case \eqref{eqn:pp}.
Expanding \eqref{eqn:pp} in a Taylor series about $\mu=0$, we obtain
\begin{equation}
\begin{split}
&
\dot{x}=p_x+y,\quad
\dot{y}=p_y-x,\\
&
\dot{p}_x=p_y-\frac{x}{(x^2+y^2)^{3/2}}+\mu\left(\frac{(x-1)(x^2+y^2)+3x^2}{(x^2+y^2)^{5/2}}
 -\frac{x-1}{(x-1)^2+y^2)^{3/2}}\right),\\
&
\dot{p}_y=-p_x-\frac{y}{(x^2+y^2)^{3/2}}+\mu\left(\frac{y(x^2+y^2)+3xy}{(x^2+y^2)^{5/2}}
-\frac{y}{((x-1)^2+y^2)^{3/2}}\right),
\end{split}
\label{eqn:pped}
\end{equation}
up to $O(\mu)$.
Equation~\eqref{eqn:pped} is a Hamiltonian system with the Hamiltonian
\begin{align}
H=&
\tfrac{1}{2}(p_x^2+p_y^2)-\frac{1}{\sqrt{x^2+y^2}}+y p_x-xp_y\notag\\
&
+\mu\left(\frac{x^2+y^2+x}{(x^2+y^2)^{3/2}}-\frac{1}{\sqrt{(x-1)^2+y^2}}\right).
\label{eqn:H2}
\end{align}

We next rewrite the Hamiltonian \eqref{eqn:H2} in the polar coordinates. 
Let
\[
x=r\cos\phi,\quad
y=r\sin\phi.
\]
The momenta $(p_r,p_\phi)$ corresponding to $(r,\phi)$ satisfy
\[
p_x=p_r\cos\phi-\frac{p_\phi}{r}\sin\phi,\quad
p_y=p_r\sin\phi+\frac{p_\phi}{r}\cos\phi.
\]
See, e.g., Section~8.6.1 of \cite{MO17}.
The Hamiltonian becomes
\begin{equation*}
H=\tfrac{1}{2}\left(p_r^2+\frac{p_\phi^2}{r^2}\right)-\frac{1}{r}-p_\phi
 +\mu\left(\frac{r+\cos\phi}{r^2}-\frac{1}{\sqrt{r^2+1-2r\cos\phi}}\right).
\end{equation*}
When $\mu=0$, the corresponding Hamiltonian system is written as
\begin{equation}
\dot{r}=p_r,\quad
\dot{p}_r=\frac{p_\phi^2}{r^3}-\frac{1}{r^2},\quad
\dot{\phi}=\frac{p_\phi}{r^2}-1,\quad
\dot{p}_\phi=0,
\label{eqn:pp0p}
\end{equation}
which is easily solved since $p_\phi$ is a constant.
Letting $\varphi=\phi+t$, we have the relation
\begin{equation}
r=\frac{p_\phi^2}{1+e\cos\varphi},
\label{eqn:ppr}
\end{equation}
where the position $\varphi=0$ is appropriately chosen and $e$ is a constant.
We choose $e\in(0,1)$, so that Eq.~\eqref{eqn:ppr}
 represents an elliptic orbit with the eccentricity $e$.
Moreover, its period is given by
\begin{equation}
T=p_\phi^3\int_0^{2\pi}\frac{\d\varphi}{(1+e\cos\varphi)^2}
 =\frac{2\pi p_\phi^3}{(1-e^2)^{3/2}}.
\label{eqn:T}
\end{equation}

Finally, we introduce  the Delaunay elements $(I_1,I_2,\theta_1,\theta_2)$
 obtained from the generating function
\begin{equation}
W(r,\phi,I_1,I_2)=I_2\phi+\chi(r,I_1,I_2),
\label{eqn:W}
\end{equation}
where
\begin{align}
\chi(r,I_1,I_2)
=&\int_{r_0}^r\left(\frac{2}{\rho}-\frac{1}{I_1^2}
 -\frac{I_2^2}{\rho^2}\right)^{1/2}\d\rho\notag\\
=&
 -2I_1^2\arcsin\sqrt{\frac{r_+-r}{r_+-r_-}}+\sqrt{(r_+-r)(r-r_-)}\notag\\
& +2I_1I_2\arctan\sqrt{\frac{r_-(r_+-r)}{r_+(r-r_-)}}
\label{eqn:chi}
\end{align}
with
\[
r_\pm=I_1\left(I_1\pm\sqrt{I_1^2-I_2^2}\right)
\]
(see, e.g., Section~8.9.1 of \cite{MO17}).
Here the upper and lower signs are taken when $p_r$ is positive and negative, respectively.
We have
\begin{equation}
\begin{split}
&
p_r=\frac{\partial W}{\partial r}
=\frac{\partial\chi}{\partial r}(r,I_1,I_2),\quad
p_\phi=\frac{\partial W}{\partial\phi}
=I_2,\\
&
\theta_1=\frac{\partial W}{\partial I_1}=\chi_1(r,I_1,I_2),\quad
\theta_2=\frac{\partial W}{\partial I_2}=\phi+\chi_2(r,I_1,I_2),
\end{split}
\label{eqn:ppDe}
\end{equation}
where
\begin{align*}
\chi_1(r,I_1,I_2)=\frac{\partial\chi}{\partial I_1}(r,I_1,I_2),\quad
\chi_2(r,I_1,I_2)=\frac{\partial\chi}{\partial I_2}(r,I_1,I_2).
\end{align*}
Since the transformation from $(r,\phi,p_r,p_\phi)$ to $(\theta_1,\theta_2,I_1,I_2)$ is symplectic,
 the transformed system is also Hamiltonian and its Hamiltonian is given by
\begin{align}
H=&-\frac{1}{2I_1^2}-I_2
 +\mu\biggl(\frac{1}{R(\theta_1,I_1,I_2)}
 +\frac{\cos(\theta_2-\chi_2(R(\theta_1,I_1,I_2),I_1,I_2))}{R(\theta_1,I_1,I_2)^2}\notag\\
&
 -\frac{1}{\sqrt{R(\theta_1,I_1,I_2)^2
-2R(\theta_1,I_1,I_2)\cos(\theta_2-\chi_2(R(\theta_1,I_1,I_2),I_1,I_2))+1}} \biggr),
\label{eqn:ppHD}
\end{align}
where $r=R(\theta_1,I_1,I_2)$ is the $r$-component of the symplectic transformation satisfying
\begin{equation}
\theta_1=\chi_1(R(\theta_1,I_1,I_2),I_1,I_2).
\label{eqn:R}
\end{equation}
The associated Hamiltonian system is written as
\begin{equation}
\begin{split}
&
\dot{I}_1=\mu h_1(I_1,I_2,\theta_1,\theta_2),\quad
\dot{I}_2=O(\mu),\\ 
&
\dot{\theta}_1=\frac{1}{I_1^3}+O(\mu),\quad
\dot{\theta}_2=-1+O(\mu),
\end{split}
\label{eqn:ppe}
\end{equation}
where
\begin{align*}
&
h_1(I_1,I_2,\theta_1,\theta_2)\\
&
=\biggl(\frac{1}{R(\theta_1,I_1,I_2)^2}
 +\frac{2\cos(\theta_2-\chi_2(R(\theta_1,I_1,I_2),I_1,I_2))}{R(\theta_1,I_1,I_2)^3}\\
&\quad
 +\frac{I_2\sin(\theta_2-\chi_2(R(\theta_1,I_1,I_2),I_1,I_2))}{R(\theta_1,I_1,I_2)^2}
 \frac{\partial\chi_2}{\partial r}(R(\theta_1,I_1,I_2),I_1,I_2))\\
&\quad
-\bigl(R(\theta_1,I_1,I_2)^2
-2R(\theta_1,I_1,I_2)\cos(\theta_2-\chi_2(R(\theta_1,I_1,I_2),I_1,I_2))+1\bigr)^{-3/2}\\
&\quad
\times\biggl(R(\theta_1,I_1,I_2)\biggl(1+I_2\sin(\theta_2-\chi_2(R(\theta_1,I_1,I_2),I_1,I_2))\\
&\quad
\times\frac{\partial\chi_2}{\partial r}(R(\theta_1,I_1,I_2),I_1,I_2))\biggr)
-\cos(\theta_2-\chi_2(R(\theta_1,I_1,I_2),I_1,I_2))\biggr)\biggr)\\
&\quad
\times\frac{\partial R}{\partial\theta_1}(\theta_1,I_1,I_2).
\end{align*}

To state our main theorem for the planar case \eqref{eqn:pp},
 we introduce  the new variables $(v_1,v_2,v_3)\in\Cset\times(\Cset/2\pi\Zset)^2$ given by
\begin{align*}
V_1(v_1,r,I_1,I_2)
 :=&v_1^2+r^2-2I_1^2r+I_1^2I_2^2=0,\\
V_2(v_2,r,I_1,I_2)
 :=&I_1^2(I_1^2-I_2^2)(2\sin^2v_2-1)^2-(r-I_1^2)^2=0,\\
V_3(v_3,r,I_1,I_2)
 :=& I_1^2(r-I_2^2)^2(\tan^2v_3+1)^2-r^2(I_1^2-I_2^2)^2(\tan^2v_3-1)^2=0,
\end{align*}
so that the generating function \eqref{eqn:W} is regarded as an analytic one
 on the four-dimensional complex manifold
\begin{align*}
\bar{\S}_2=\{(r,\phi,I_1,I_2,v_1,v_2,v_3)
&\in\Cset\times(\Cset/2\pi\Zset)\times\Cset^3\times(\Cset/2\pi\Zset)^2\\
& \mid V_j(v_j,r,I_1,I_2)=0,\ j=1,2,3\}
\end{align*}
since Eq.~\eqref{eqn:chi} is represented by
\[
\chi(I_1,I_2,v_1,v_2,v_3)=v_1-2I_1^2v_2+2I_1I_2v_3.
\]
Hence, we can regard \eqref{eqn:ppe}
 as a meromorphic two-degree-of-freedom Hamiltonian systems
 on the four-dimensional complex manifold
\begin{align*}
\S_2=&\{(I_1,I_2,\theta_1,\theta_2,r,v_1,v_2,v_3)
\in\Cset^2\times(\Cset/2\pi\Zset)^2\times\Cset^2\times(\Cset/2\pi\Zset)^2\\
& \quad
 \mid\theta_1-\chi_1(r,I_1,I_2)=V_j(v_j,r,I_1,I_2)=0,\ j=1,2,3\}.
\end{align*}
Actually, we have
\begin{align*}
\frac{\partial V_j}{\partial v_j}\frac{\partial v_j}{\partial r}
 +\frac{\partial V_j}{\partial r}=0,\quad
\frac{\partial V_j}{\partial v_j}\frac{\partial v_j}{\partial I_l}
 +\frac{\partial V_j}{\partial I_l}=0,\quad
 j=1,2,3,\
 l=1,2,
\end{align*}
to express
\begin{align*}
\frac{\partial\chi}{\partial r}
 =\sum_{j=1}^3\frac{\partial\chi}{\partial v_j}\frac{\partial V_j}{\partial r},\quad
\chi_l=\frac{\partial\chi}{\partial I_l}
 +\sum_{j=1}^3\frac{\partial\chi}{\partial v_j}\frac{\partial V_j}{\partial I_l},\quad
 l=1,2
\end{align*}
as meromorophic functions of $(r,I_1,I_2,v_1,v_2,v_3)$ on $\bar{\S}_2$.
Similar treatments were used for Hamiltonian systems with algebraic potentials originally in \cite{C13}
 and for the restricted three-body problem with $\mu\in(0,1)$ fixed in \cite{Y21a}.
Let $\Sigma(\S_2)$ be the critical set of $\S_2$
 on which the projection $\pi_2:\S_2\to\Cset^2\times(\Cset/2\pi\Zset)^2$ given by
\[
\pi_2(I_1,I_2,\theta_1,\theta_2,r,v_1,v_2,v_3)=(I_1,I_2,\theta_1,\theta_2)
\]
is singular.

\begin{thm}
\label{thm:pp}
The Hamiltonian system \eqref{eqn:ppe} with the Hamiltonian \eqref{eqn:ppHD}
 does not have a complete set of first integrals in involution
 that are functionally independent almost everywhere
 and meromorphic in $(I_1,I_2,\theta_1,\theta_2,v_1,v_2,v_3,\mu)$
in a neighborhood of the unperturbed orbit $\{I_1,I_2=\text{const.}\}$
 with $I_2\in(0,I_1)$ or $(I_1,0)$ and $I_1^3\in\Qset\setminus\{0\}$
 on $\S_2\setminus\Sigma(\S_2)$ near $\mu=0$.
\end{thm}

\begin{proof}
We only have to show that
 the hypotheses of Theorem~\ref{thm:tool} hold for \eqref{eqn:ppe} as in Theorem~2 of \cite{C13},
 since the corresponding Hamiltonian system has the same expression as \eqref{eqn:ppe}
 on $\S_2\setminus\Sigma(\S_2)$.

We first estimate $h_1(I,\theta)$ for the unperturbed solutions to \eqref{eqn:ppe}.
When $\mu=0$, we see that $I_1,I_2$ are constants
 and can write $\theta_1=\omega_1 t+\theta_{10}$ and $\theta_2=-t+\theta_{20}$
 for any solution to \eqref{eqn:ppe}, where
\begin{equation}
\omega_1=\frac{1}{I_1^3}
\label{eqn:omega1}
\end{equation}
and $\theta_{10},\theta_{20}\in\Sset^1$ are constants.
Since $r=R(\omega_1 t+\theta_{10},I_1,I_2)$ and
\[
\phi=-t+\chi_2(R(\omega_1 t+\theta_{10},I_1,I_2),I_1,I_2),
\]
respectively, become the $r$- and $\phi$-components of a solution to \eqref{eqn:pp0p},
 we have
\begin{equation}
\begin{split}
&
R(\omega_1 t+\theta_{10},I_1,I_2)
 =\frac{I_2^2}{1+e\cos(\phi(t)+t+\bar{\phi}(\theta_{10}))},\\
&
\chi_2(R(\omega_1 t+\theta_{10},I_1,I_2),I_1,I_2)
 =\phi(t)+t+\bar{\phi}(\theta_{10})
\end{split}
\label{eqn:pprchi}
\end{equation}
by \eqref{eqn:ppr},
 where $\phi(t)$ is the $\phi$-component of a solution to \eqref{eqn:pp0p}
  and $\bar{\phi}(\theta_{10})$ is a constant depending on $\theta_{10}$.
Differentiating both equations in \eqref{eqn:pprchi} with respect to $t$ yields
\begin{equation}
\begin{split}
&
\omega_1\frac{\partial R}{\partial\theta_1}(\omega_1t+\theta_{10},I_1,I_2)
=\frac{e I_2^2\sin(\phi(t)+t+\bar{\phi}(\theta_{10}))(\dot{\phi}(t)+1)}
 {(1+e\cos(\phi(t)+t+\bar{\phi}(\theta_{10})))^2},\\
&
\omega_1 I_2\frac{\partial\chi_2}{\partial r}
(R(\omega_1 t+\theta_{10},I_1,I_2),I_1,I_2)
\frac{\partial R}{\partial\theta_1}(\omega_1t+\theta_{10},I_1,I_2)
=\dot{\phi}(t)+1.
\end{split}
\label{eqn:pprchi2}
\end{equation}

Fix the value of $I_1=I_1^\ast$ such that $\omega_1=k_1/k_2$ for $k_1,k_2>0$ relatively prime integers,
 i.e., $I_1^{\ast 3}=k_2/k_1\in\Qset$.
The following arguments can be modified to apply when $\omega_1<0$.
Assumption~(A1) holds with $\omega^\ast=1/k_2$.
By the second equation of \eqref{eqn:pprchi} $\phi(t)$ is $2\pi k_2/k_1$-periodic.
From \eqref{eqn:T} and \eqref{eqn:omega1} we have
\begin{equation}
\frac{2\pi I_2^3}{(1-e^2)^{3/2}}=\frac{2\pi k_2}{k_1},\quad\mbox{i.e.,}\quad
I_2=I_1^\ast\sqrt{1-e^2}=:I_2^\ast
\label{eqn:con}
\end{equation}
and
\[
\D\omega(I^\ast)
=\begin{pmatrix}
-3/I_1^{\ast 4} & 0\\
0 & 0
\end{pmatrix}
=\begin{pmatrix}
-3(k_1/k_2)^{4/3} & 0\\
0 & 0
\end{pmatrix}.
\]
Using \eqref{eqn:pprchi} and \eqref{eqn:pprchi2}, we obtain
\begin{align}
&
h_1(I_1^\ast,I_2^\ast,t+\theta_1,-t+\theta_2)\notag\\
&
=\frac{k_2}{k_1}\dot{\varphi}(t)\biggl(\frac{e}{I_2^{\ast 2}}\sin\varphi(t)
 +\frac{2e}{I_2^{\ast 4}}(1+e\cos\varphi(t))\sin\varphi(t)\cos(\phi(t)+\theta_2)\notag\\
&\quad
+\frac{1}{I_2^{\ast 4}}(1+e\cos\varphi(t))^2\sin(\phi(t)+\theta_2)\notag\\
&\quad
 -\frac{1}{\hat{r}(t;\theta_2)^3}\biggl(e I_2^{\ast 4}\sin\varphi(t)
 +I_2^{\ast 2}(1+e\cos\varphi(t))^2\sin(\phi(t)+\theta_2)\notag\\
&\quad
 -e I_2^{\ast 2}(1+e\cos\varphi(t))\sin\varphi(t)\cos(\phi(t)+\theta_2)\biggr)
\label{eqn:ph1}
\end{align}
after an appropriate shift of the time variable $t$, where $\varphi(t)=\phi(t)+t$ and
\begin{equation}
\hat{r}(t;\theta_2)=\sqrt{(1+e\cos\varphi(t))^2
 -2I_2^{\ast 2}(1+e\cos\varphi(t))\cos(\phi(t)+\theta_2)+I_2^{\ast 4}}.
\label{eqn:hatr}
\end{equation}
Since by \eqref{eqn:pp0p} and \eqref{eqn:ppr}
\begin{equation}
\frac{\dot{\varphi}(t)}{(1+e\cos\varphi(t))^2}=\frac{1}{I_2^{\ast 3}}
=\frac{k_1}{k_2(1-e^2)^{3/2}},
\label{eqn:ppdphi}
\end{equation}
we write the first component of \eqref{eqn:A2} as
\[
\mathscr{I}_1(\theta)
=-\frac{3k_2}{k_1I_1^{\ast 4}}\int_{\gamma_\theta}\tilde{h}_1(t;\theta_2)\d t,
\]
where
\begin{align*}
\tilde{h}_1(t;\theta_2)=&
 \frac{e}{I_2^{\ast 5}}(1+e\cos\varphi(t))^2\sin\varphi(t)\\
&
+\frac{2e}{I_2^{\ast 7}}(1+e\cos\varphi(t))^3\sin\varphi(t)\cos(\phi(t)+\theta_2)\notag\\
&
+\frac{1}{I_2^{\ast 7}}(1+e\cos\varphi(t))^4\sin(\phi(t)+\theta_2)\notag\\
&
 -\frac{1}{\hat{r}(t;\theta_2)^3}\biggl(e I_2^\ast(1+e\cos\varphi(t))^2\sin\varphi(t)\\
&
 +\frac{1}{I_2^\ast}(1+e\cos\varphi(t))^4\sin(\phi(t)+\theta_2)\notag\\
&
 -\frac{e}{I_2^\ast}(1+e\cos\varphi(t))^3\sin\varphi(t)\cos(\phi(t)+\theta_2)\biggr)
\end{align*}
and the closed loop $\gamma_\theta$ is specified below.

\begin{figure}[t]
\includegraphics[scale=0.5]{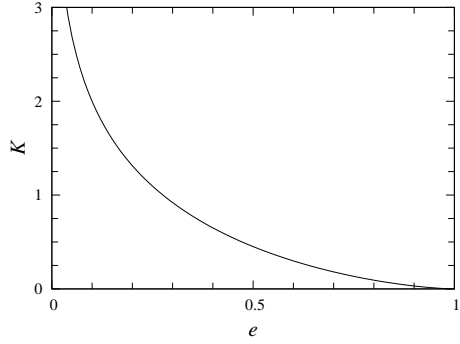}
\caption{Dependence of $K$ on $e$.
 \label{fig:3a}}
\end{figure}

We integrate \eqref{eqn:ppdphi} to obtain
\begin{equation}
\frac{k_1 t}{k_2}=2\arctan\biggl(\frac{(1-e)\tan\tfrac{1}{2}\varphi}{\sqrt{1-e^2}}\biggr)
 -\frac{e\sqrt{1-e^2}\sin\varphi}{1+e\cos\varphi}+\pi
\label{eqn:varphi}
\end{equation}
when $\varphi(k_2\pi/k_1)=0$.
As $\im\varphi\to\infty$, we have $t\to k_2(\pi+ iK)/k_1$ and
\begin{align*}
\e^{-2i\varphi}\left(t-\frac{k_2}{k_1}(\pi+iK)\right)
=-\frac{2i(1-e^2)^{3/2}}{e^2}+o(1),
\end{align*}
where
\[
K=2\arctanh\biggl(\frac{1-e}{\sqrt{1-e^2}}\biggr)-\sqrt{1-e^2}>0.
\]
See Fig.~\ref{fig:3a}.
Thus, we obtain
\begin{align*}
&
\e^{-i\varphi}
=-\frac{(1-i)(1-e^2)^{3/4}}{e(t-k_2(\pi+iK)/k_1)^{1/2}}+O(1)
\end{align*}
as $t\to\pi+iK$.
We take a small circle with radius $\delta>0$ and center at $t=\pi+iK$ as $\gamma_\theta$.
On the circle, we estimate
\begin{align*}
&
(1+e\cos\varphi(t))^3\sin\varphi(t)\cos(\varphi(t)-t+\theta_2)\\
&
=\frac{(1+i)(1-e^2)^{15/4}\e^{-(k_2K/k_1+i(\theta_2-k_2\pi/k_1))}}{8e^2\zeta^{5/2}}+o(|\zeta|^{-5/2})
\end{align*}
and
\begin{align*}
&
(1+e\cos\varphi(t))^4\sin(\varphi(t)-t+\theta_2)\\
&=\frac{(1+i)(1-e^2)^{15/4}\e^{-(k_2K/k_1+i(\theta_2-k_2\pi/k_1))}}{8e\zeta^{5/2}}+o(|\zeta|^{-5/2})
\end{align*}
where $\zeta+iK\in\gamma_\theta$.
Noting that the other terms in the integrand are $o(|\zeta|^{-5/2})$ on $\gamma_\theta$,
 we compute
\begin{align*}
\mathscr{I}_1(\theta)
=& -\frac{9(1-i)k_1^2(1-e^2)^{1/4}}{4ek_2^2\delta^{3/2}}
 \e^{-(k_2K/k_1+i(\theta_2-k_2\pi/k_1))}+o(\delta^{-3/2}),
\end{align*}
so that condition~(A2) holds.
Here we have used  the relations $I_1^\ast=(k_2/k_1)^{1/3}$ and $I_2^\ast=(k_2/k_1)^{1/3}\sqrt{1-e^2}$.
Using Theorem~\ref{thm:tool},
 we obtain the desired result.
\end{proof}

\begin{rmk}
\label{rmk:3a}
From the above proof we see that
 the planar problem \eqref{eqn:pp} is nonintegrable near the periodic orbits
\[
r=\frac{(k_2/k_1)^{2/3}(1-e^2)}{1+e\cos(\varphi(t)+\bar{\phi})},\quad
t\in\Cset,
\]
for any $e\in(0,1)$, $\bar{\phi}\in\Sset^1$ and $(k_1,k_2)$ of relatively prime positive integers,
 where $\varphi(t)$ is implicitly given by \eqref{eqn:varphi}
 $($see the first equation of \eqref{eqn:pprchi} and Eq.~\eqref{eqn:con}$)$.
\end{rmk}

\section{Spatial Case}

We next discuss the spatial case \eqref{eqn:sp}.
Expanding \eqref{eqn:sp} in a Taylor series about $\mu=0$, we obtain
\begin{equation}
\begin{split}
&
\dot{x}=p_x+y,\quad
\dot{y}=p_y-x,\quad
\dot{z}=p_z,\\
&
\dot{p}_x=p_y-\frac{x}{(x^2+y^2+z^2)^{3/2}}\\
&\qquad
+\mu\left(\frac{(x-1)(x^2+y^2+z^2)+3x^2}{(x^2+y^2+z^2)^{5/2}}
 -\frac{x-1}{(x-1)^2+y^2+z^2)^{3/2}}\right),\\
&
\dot{p}_y=-p_x-\frac{y}{(x^2+y^2+z^2)^{3/2}}\\
&\qquad
 +\mu\left(\frac{y(x^2+y^2+z^2)+3xy}{(x^2+y^2+z^2)^{5/2}}
-\frac{y}{((x-1)^2+y^2+z^2)^{3/2}}\right),\\
&
\dot{p}_z=-\frac{z}{(x^2+y^2+z^2)^{3/2}}\\
&\qquad
 +\mu\left(\frac{y(x^2+y^2+z^2)+3zx}{(x^2+y^2+z^2)^{5/2}}
-\frac{z}{((x-1)^2+y^2+z^2)^{3/2}}\right),
\end{split}
\label{eqn:sped}
\end{equation}
up to $O(\mu)$.
Equation~\eqref{eqn:sped} is a Hamiltonian system with the Hamiltonian
\begin{align}
H=&
\tfrac{1}{2}(p_x^2+p_y^2+p_z^2)-\frac{1}{\sqrt{x^2+y^2+z^2}}+y p_x-xp_y\notag\\
&
+\mu\left(\frac{x^2+y^2+z^2+x}{(x^2+y^2)^{3/2}}-\frac{1}{\sqrt{(x-1)^2+y^2+z^2}}\right).
\label{eqn:H3}
\end{align}

\begin{figure}[t]
\includegraphics[scale=0.8]{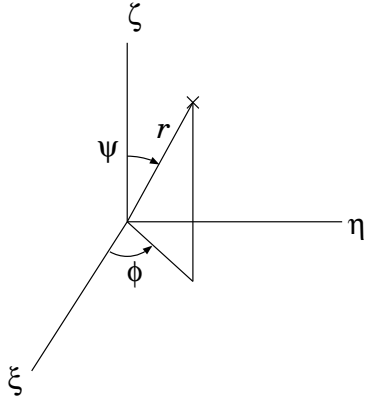}
\caption{Spherical coordinates.
 \label{fig:4a}}
\end{figure}

We next rewrite the Hamiltonian \eqref{eqn:H3} in the spherical coordinates.
Let
\[
x=r\sin\psi\cos\phi,\quad
y=r\sin\psi\sin\phi,\quad
z=r\cos\psi.
\]
See Fig.~\ref{fig:4a}.
The momenta $(p_r,p_\phi,p_\psi)$ corresponding to $(r,\phi,\psi)$ satisfy
\begin{align*}
&
p_x=p_r\cos\phi\sin\psi-\frac{p_\phi\sin\phi}{r\sin\psi}+\frac{p_\psi}{r}\cos\phi\cos\psi,\\
&
p_y=p_r\sin\phi\sin\psi+\frac{p_\phi\cos\phi}{r\sin\psi}+\frac{p_\psi}{r}\sin\phi\cos\psi,\\
&
p_z=p_r\cos\psi-\frac{p_\psi}{r}\sin\psi
\end{align*}
(see, e.g., Section~8.7 of \cite{MO17}).
The Hamiltonian becomes
\begin{align*}
H=&\tfrac{1}{2}\left(p_r^2+\frac{p_\psi^2}{r^2}+\frac{p_\phi^2}{r^2\sin^2\psi}\right)-\frac{1}{r}-p_\phi\\
& +\mu\biggl(\frac{r+\sin\psi\cos\phi}{r^2}-\frac{1}{\sqrt{r^2+1-2r\sin\psi\cos\phi}}\biggr).
\end{align*}
When $\mu=0$, the corresponding Hamiltonian system is written as
\begin{equation}
\begin{split}
&
\dot{r}=p_r,\quad
\dot{p}_r=\frac{p_\psi^2}{r^3}+\frac{p_\phi^2}{r^3\sin^2\psi}-\frac{1}{r^2},\quad
\dot{\phi}=\frac{p_\phi}{r^2\sin^2\psi}-1,\\
&
\dot{p}_\phi=0,\quad
\dot{\psi}=\frac{p_\psi}{r^2},\quad
\dot{p}_\psi=\frac{p_\phi^2\cos\psi}{r^2\sin^3\psi}.
\label{eqn:sp0p}
\end{split}
\end{equation}
We also have the relation \eqref{eqn:ppr} for periodic orbits on the $xy$-plane
 since Eq.~\eqref{eqn:sp0p} reduces to \eqref{eqn:pp0p} there
 when $\psi=\tfrac{1}{2}\pi$ and $p_\psi=0$.

As in the planar case,
 we introduce  the Delaunay elements obtained from the generating function
\begin{equation}
\hat{W}(r,\phi,\psi,I_1,I_2.I_2)=I_3\phi+\chi(r,I_1,I_2)+\hat{\chi}(\psi,I_2,I_3),
\label{eqn:hW}
\end{equation}
where
\begin{align}
\hat{\chi}(\psi,I_2,I_3)
=&\int_{\psi_0}^\psi\left(I_2^2-\frac{I_3^2}{\sin^2s}\right)^{1/2}\d s\notag\\
=& I_2\arctan\frac{\sqrt{I_2^2\sin^2\psi-I_3^2}}{I_2\cos\psi}
 -I_3\arctan\frac{\sqrt{I_2^2\sin^2\psi-I_3^2}}{I_3\cos\psi}
\label{eqn:hchi}
\end{align}
with $\psi_0=\arcsin(I_3/I_2)$.
See, e.g., Section~8.9.3 of \cite{MO17},
 although a slightly modified generating function is used here.
We have
\begin{equation}
\begin{split}
&
p_r=\frac{\partial \hat{W}}{\partial r}
=\frac{\partial\chi}{\partial r}(r,I_1,I_2),\quad
p_\phi=\frac{\partial \hat{W}}{\partial\phi}
=I_3,\\
&
p_\psi=\frac{\partial\hat{W}}{\partial\psi}
=\frac{\partial\hat{\chi}}{\partial\psi}(\psi,I_2,I_3),\quad
\theta_1=\frac{\partial \hat{W}}{\partial I_1}
=\chi_1(r,I_1,I_2),\\
&
\theta_2=\frac{\partial\hat{W}}{\partial I_2}
=\chi_2(r,I_1,I_2)+\hat{\chi}_2(\psi,I_2,I_3),\quad
\theta_3=\frac{\partial \hat{W}}{\partial I_3}
=\phi+\hat{\chi}_3(\psi,I_2,I_3),
\end{split}
\label{eqn:spDe}
\end{equation}
where
\begin{align*}
\hat{\chi}_2(\psi,I_2,I_3)
=\frac{\partial\hat{\chi}}{\partial I_2}(\psi,I_2,I_3),\quad
\hat{\chi}_3(\psi,I_2,I_3)
=\frac{\partial\hat{\chi}}{\partial I_3}(\psi,I_2,I_3).
\end{align*}
Since the transformation from $(r,\phi,\psi,p_r,p_\phi,p_\psi)$
 to $(\theta_1,\theta_2,\theta_3,I_1,I_2,I_3)$ is symplectic,
 the transformed system is also Hamiltonian and its Hamiltonian is given by
\begin{align}
H=&-\frac{1}{2I_1^2}-I_3+\mu\biggl(\frac{1}{R(\theta_1,I_1,I_2)}
 +\frac{\sin\Psi(\theta_1,\theta_2,I_1,I_2,I_3)}{R(\theta_1,I_1,I_2)^2}\notag\\
& \times\cos(\theta_3-\hat{\chi}_3(\Psi(\theta_1,\theta_2,I_1,I_2,I_3),I_2,I_3))\notag\\
& -\bigl(R(\theta_1,I_1,I_2)^2+1-2R(\theta_1,I_1,I_2)\sin\Psi(\theta_1,\theta_2,I_1,I_2,I_3)\notag\\
& \times\cos(\theta_3-\hat{\chi}_3(\Psi(\theta_1,\theta_2,I_1,I_2,I_3),I_2,I_3))\bigr)^{-1/2}\biggr),
\label{eqn:spHD}
\end{align}
where $r=R(\theta_1,I_1,I_2)$ and $\psi=\Psi(\theta_1,\theta_2,I_1,I_2,I_3)$
 are the $r$- and $\psi$-components of the symplectic transformation
 satisfying \eqref{eqn:R} and
\[
\hat{\chi}_2(\Psi(\theta_1,\theta_2,I_1,I_2,I_3),I_2,I_3)+\chi_2(R(\theta_1,I_1,I_2),I_1,I_2))
 =\theta_2,
\]
respectively.
The associated Hamiltonian system is written as
\begin{equation}
\begin{split}
&
\dot{I}_1=\mu h_1(I_1,I_2,I_3,\theta_1,\theta_2,\theta_3),\quad
\dot{I}_2=O(\mu),\quad 
\dot{I}_3=O(\mu),\\
&
\dot{\theta}_1=\frac{1}{I_1^3}+O(\mu),\quad
\dot{\theta}_2=O(\mu),\quad
\dot{\theta}_3=-1+O(\mu),
\end{split}
\label{eqn:spe}
\end{equation}
where
\begin{align*}
&
h_1(I_1,I_2,I_3,\theta_1,\theta_2,\theta_3)\\
&
=\frac{\partial R}{\partial\theta_1}(\theta_1,I_1,I_2)\biggl(\frac{1}{R(\theta_1,I_1,I_2)^2}\\
&\quad
+\frac{2\sin\Psi(\theta_1,\theta_2,I_1,I_2,I_3)}{R(\theta_1,I_1,I_2)^3}
 \cos(\theta_3-\hat{\chi}_3(\Psi(\theta_1,\theta_2,I_1,I_2,I_3),I_2,I_3))\\
&\quad
-\bigl(R(\theta_1,I_1,I_2)^2+1-2R(\theta_1,I_1,I_2)\sin\Psi(\theta_1,\theta_2,I_1,I_2,I_3)\\
&\quad
\times\cos(\theta_3-\hat{\chi}_3(\Psi(\theta_1,\theta_2,I_1,I_2,I_3),I_2,I_3))\bigr)^{-3/2}
 \biggl(R(\theta_1,I_1,I_2)\\
&\quad
-\sin\Psi(\theta_1,\theta_2,I_1,I_2,I_3)
 \cos(\theta_3-\hat{\chi}_3(\Psi(\theta_1,\theta_2,I_1,I_2,I_3),I_2,I_3))\biggr)\biggr)\\
&\quad
+\frac{\partial\Psi}{\partial\theta_1}(\theta_1,\theta_2,I_1,I_2,I_3)
\biggl(\frac{\sin\Psi(\theta_1,\theta_2,I_1,I_2,I_3)}{R(\theta_1,I_1,I_2)^2}\\
&\quad
\times\sin(\theta_3-\hat{\chi}_3(\Psi(\theta_1,\theta_2,I_1,I_2,I_3),I_2,I_3))
 \frac{\partial\hat{\chi}_3}{\partial\psi}(\Psi(\theta_1,\theta_2,I_1,I_2,I_3),I_2,I_3))\\
&\quad
 -\frac{\cos\Psi(\theta_1,\theta_2,I_1,I_2,I_3)}{R(\theta_1,I_1,I_2)^2}
 \cos(\theta_3-\hat{\chi}_3(\Psi(\theta_1,\theta_2,I_1,I_2,I_3),I_2,I_3))\\
&\quad
+R(\theta_1,I_1,I_2)\bigl(R(\theta_1,I_1,I_2)^2+1-2R(\theta_1,I_1,I_2)\sin\Psi(\theta_1,\theta_2,I_1,I_2,I_3)\\
&\quad
\times\cos(\theta_3-\hat{\chi}_3(\Psi(\theta_1,\theta_2,I_1,I_2,I_3),I_2,I_3))\bigr)^{-3/2}
 \biggl(\cos\Psi(\theta_1,\theta_2,I_1,I_2,I_3)\\
&\quad
\times\cos(\theta_3-\hat{\chi}_3(\Psi(\theta_1,\theta_2,I_1,I_2,I_3),I_2,I_3))
 -\sin\Psi(\theta_1,\theta_2,I_1,I_2,I_3)\\
&\quad
\times\sin(\theta_3-\hat{\chi}_3(\Psi(\theta_1,\theta_2,I_1,I_2,I_3),I_2,I_3))
 \frac{\partial\hat{\chi}_3}{\partial\psi}(\Psi(\theta_1,\theta_2,I_1,I_2,I_3),I_2,I_3))\biggr).
\end{align*}

As in the planar case, we introduce the new variables $w_1,w_2\in(\Cset/2\pi)$ given by
\begin{align*}
&
W_1(w_1,\psi,I_2,I_3):=I_2^2\cos^2\psi\tan^2w_1-I_2^2\sin^2\psi+I_3^2=0,\\
&
W_2(w_2,\psi,I_2,I_3):=I_3^2\cos^2\psi\tan^2w_2-I_2^2\sin^2\psi+I_3^2=0
\end{align*}
are introduced,
 so that the generating function \eqref{eqn:hW} is regarded as an analytic one
 on the six--dimensional complex manifold
\begin{align*}
\bar{\S}_3=&\{(r,\phi,\psi,I_1,I_2,I_3,v_1,v_2,v_3,w_1,w_2)
\in\Cset\times(\Cset/2\pi\Zset)^2\times\Cset^4\times(\Cset/2\pi\Zset)^4\\
&\quad
\mid V_j(v_j,r,I_1,I_2)=W_l(w_l,\psi,I_2,I_3)=0,\ j=1,2,3,\ l=1,2\}
\end{align*}
since Eq.~\eqref{eqn:hchi} is represented by
\[
\hat{\chi}(I_2,I_3;v_1,v_2)=I_2w_1-I_3w_2.
\]
Moreover, we can regard \eqref{eqn:spe}
 as a meromorphic three-degree-of-freedom Hamiltonian systems
 on the six-dimensional complex manifold
\begin{align*}
\S_3=&\{(I_1,I_2,I_3,\theta_1,\theta_2,\theta_3,r,v_1,v_2,v_3,w_1,w_2)
 \in\Cset^3\times(\Cset/2\pi\Zset)^3\times\Cset^2\times(\Cset/2\pi\Zset)^4\\
&\quad
 \mid \theta_1-\chi_1(r,I_1,I_2)=V_j(v_j,r,I_1,I_2)=W_l(w_l,\psi,I_2,I_3)=0,\\
&\qquad
  j=1,2,3,\ l=1,2\}.
\end{align*}
Actually, we have
\begin{align*}
\frac{\partial W_j}{\partial w_j}\frac{\partial w_j}{\partial\psi}
 +\frac{\partial W_j}{\partial\psi}=0,\quad
\frac{\partial W_j}{\partial w_j}\frac{\partial w_j}{\partial I_l}
 +\frac{\partial W_j}{\partial I_l}=0,\quad
 j=1,2,\ l=2,3
\end{align*}
to express
\begin{align*}
\frac{\partial\hat{\chi}}{\partial\psi}
 =\sum_{j=1}^2\frac{\partial\hat{\chi}}{\partial w_j}\frac{\partial W_j}{\partial\psi},\quad
\hat{\chi}_l=\frac{\partial\hat{\chi}}{\partial I_l}
 +\sum_{j=1}^2\frac{\partial\hat\chi}{\partial w_j}\frac{\partial W_j}{\partial I_l},\quad
 l=2,3
\end{align*}
as meromorophic functions of $(\psi,I_2,I_3,w_1,w_2)$ on $\bar{\S}_3$.
Let $\Sigma(\S_3)$ is the critical set of $\S_3$
 on which the projection $\pi_3:\hat{\S}_3\to\Cset^3\times(\Cset/2\pi\Zset)^3$
 given by
\[
\pi_3(I_1,I_2,I_3,\theta_1,\theta_2,\theta_3,r,v_1,v_2,v_3,w_1,w_2)
 =(I_1,I_2,I_3,\theta_1,\theta_2,\theta_3)
\]
is singular.

\begin{thm}
\label{thm:sp}
The Hamiltonian system \eqref{eqn:ppe} with the Hamiltonian \eqref{eqn:ppHD}
 does not have a complete set of first integrals in involution
 that are functionally independent almost everywhere
 and meromorphic in $(I_1,I_2,I_3,\theta_1,\theta_2,\theta_3,v_1,v_2,v_3,w_1,w_2,\mu)$
 in a neighborhood of the unperturbed orbit $\{I_1,I_2,I_3=\text{const.}\}$
 with $I_2=I_3\in(0,I_1)$ or $(I_1,0)$ and $I_1^3\in\Qset\setminus\{0\}$
 on $\S_3\setminus\Sigma(\S_3)$ near $\mu=0$.
\end{thm}

\begin{proof}
We only have to show that
 the hypotheses of Theorem~\ref{thm:tool} hold for \eqref{eqn:spe} as in the proof of Theorem~\ref{thm:pp}.

We first estimate $h_1(I,\theta)$ for the unperturbed solutions to \eqref{eqn:spe}.
When $\mu=0$, we see that $I_1,I_2,I_3,\theta_2$ are constants
 and can write $\theta_1=\omega_1 t+\theta_{10}$ and $\theta_3=-t+\theta_{30}$
 for any solution to \eqref{eqn:spe} with \eqref{eqn:omega1},
 where $\theta_{10},\theta_{30}\in\Sset^1$ are constants.
Note that if $\psi=\tfrac{1}{2}\pi$ and $p_\psi=0$, then $I_2=I_3$ by \eqref{eqn:spDe}.
Since $r=R(\omega_1 t+\theta_{10},I_1,I_2)$ and
\[
\phi=-t-\hat{\chi}_3(\Psi(\omega_1 t+\theta_{10},\theta_2,I_1,I_2),I_2,I_2),\quad
\Psi(\omega_1 t+\theta_{10},\theta_2,I_1,I_2)=\tfrac{1}{2}\pi,
\]
respectively, become the $r$- and $\phi$-components of a solution to \eqref{eqn:sp0p}
 with $\psi=\tfrac{1}{2}\pi$ and $p_\psi=0$, we have the first equation of \eqref{eqn:pprchi} with
\begin{equation}
-\hat{\chi}_3(\Psi(\omega_1 t+\theta_{10},
 \theta_2,I_1,I_2,I_2),I_2,I_2)
 =\phi(t)+t+\bar{\phi}(\theta_{10}),
\label{eqn:sprchi}
\end{equation}
where $\phi(t)$ is the $\phi$-component of a solution to \eqref{eqn:pp0p}
 and $\bar{\phi}(\theta_1)$ is a constant depending only on $\theta_1$ as in the planar case.
Differentiating \eqref{eqn:sprchi} with respect to $t$ yields
\begin{align}
&
-\omega_1\frac{\partial \hat{\chi}_3}{\partial\psi}
 (\Psi(\omega_1 t+\theta_{10},
 \theta_2,I_1,I_2,I_2),I_2,I_2)\notag\\
&\qquad
\times\frac{\partial\Psi}{\partial\theta_1}(\omega_1 t+\theta_{10},
 \theta_2,I_1,I_2,I_2)=\dot{\phi}(t)+1.
\label{eqn:sprchi2}
\end{align}

Fix $I_1=I_1^\ast$ such that $\omega_1=k_1/k_2$ for $k_1,k_2>0$ relatively prime integers,
 i.e., $I_1^{\ast 3}=k_2/k_1\in\Qset$.
The following arguments can be modified to apply when $\omega_1<0$, as in Section~3.
Assumption~(A1) holds with $\omega^\ast=1/k_2$ as in Section~3.
By \eqref{eqn:sprchi} $\phi(t)$ is $2\pi k_2/k_1$-periodic, so that by \eqref{eqn:T}
\[
I_2=I_3=I_1^\ast\sqrt{1-\epsilon^2}\ (=:I_2^\ast).
\]
From \eqref{eqn:omega1} we have
\[
\D\omega(I^\ast)=
\begin{pmatrix}
-3/I_1^{\ast 4} & 0 & 0\\
0 & 0 & 0\\
0 & 0 & 0
\end{pmatrix},
\]
where $I=I^\ast:=(I_1^\ast,I_2^\ast,I_2^\ast)$.
Using the first equations of \eqref{eqn:pprchi} and \eqref{eqn:pprchi2},
 \eqref{eqn:sprchi} and \eqref{eqn:sprchi2}, we obtain
\begin{align}
&
h_1(I^\ast,\tau+\theta_1,\theta_2,-\tau+\theta_3)\d\tau\notag\\
&
=\frac{k_2\dot{\varphi}(t)}{k_1}\biggl(\frac{e}{I_2^{\ast 2}}\sin\varphi(t)
 +\frac{2e}{I_2^{\ast 4}}(1+e\cos\varphi(t))\sin\varphi(t)\cos(\phi(t)+\theta_3)\notag\\
&\quad
+\frac{1}{I_2^{\ast 4}}(1+e\cos\varphi(t))^2\sin(\phi(t)+\theta_3)\notag\\
&\quad
 -\frac{1}{\hat{r}(t;\theta_3)^3}\biggl(e I_2^{\ast 4}\sin\varphi(t)
 +I_2^{\ast 2}(1+e\cos\varphi(t))^2\sin(\phi(t)+\theta_3\notag\\
&\quad
 -e I_2^{\ast 2}(1+e\cos\varphi(t))\sin\varphi(t)\cos(\phi(t)+\theta_3)\biggr)
\label{eqn:sh1}
\end{align}
after an appropriate shift of the time variable $t$,
 where 
 $\hat{r}(t;\theta_3)$ is given by \eqref{eqn:hatr} with $\theta_2=\theta_3$.
Equation \eqref{eqn:sh1}
 has the same expression as \eqref{eqn:ph1} with $\theta_2=\theta_3$.
Repeating the arguments in the proof of Theorem~\ref{thm:pp},
 we can show that assumption~(A2) holds as in the planar case.
Thus, we use Theorem~\ref{thm:tool} to complete the proof.
\end{proof}

\begin{rmk}
\label{rmk:4a}
From the above proof we see that
 the planar problem \eqref{eqn:sp} is nonintegrable near the periodic orbits
\[
r=\frac{(k_2/k_1)^{2/3}(1-e^2)}{1+e\cos(\varphi(t)+\bar{\phi})},\quad
t\in\Cset,
\]
on the $xy$-plane,
for any $e\in(0,1)$, $\bar{\phi}\in\Sset^1$ and $(k_1,k_2)$ of relatively prime positive integers,
 where $\varphi(t)$ is implicitly given by \eqref{eqn:varphi},
 as in Remark~$\ref{rmk:3a}$.
\end{rmk}

\section*{Data Availability}
Data sharing not applicable to this article as no dataset was generated or analyzed
 during the current study.


%


\end{document}